\documentclass[12pt]{article}
\usepackage{amsmath}
\usepackage{amssymb}
\usepackage{amsfonts}
\usepackage{amsmath,amsthm,amssymb}
\allowdisplaybreaks

\newcommand{\ds}{\displaystyle}

\theoremstyle{plain}
\newtheorem{theorem}{Theorem}[section]
\newtheorem{remark}[theorem]{Remark}
\newtheorem{lemma}[theorem]{Lemma}
\newtheorem{proposition}[theorem]{Proposition}
\newtheorem{corollary}[theorem]{Corollary}

\def \[{\begin{equation}}
\def \]{\end{equation}}

\newif \ifLastSection \LastSectionfalse

\numberwithin{equation}{section}

\newcommand{\R}{{\mathbb R}}

\begin{document}
\title{   On the concentration phenomenon of $L^2$-subcritical constrained minimizers for a class of Kirchhoff equations with potentials\thanks{a: Partially supported by NSFC NO: 11501428, NSFC NO: 11371159.}}
\author{Gongbao Li,~~~~~~Hongyu Ye \thanks{b:
E-mail address: ligb@mail.ccnu.edu.cn, yyeehongyu@163.com}  \\ \small { School of Mathematics and
Statistics, Central China  Normal University,}\\
\small {  Wuhan, 430079, P. R. China}\\
\small {College of Science, Wuhan University of Science and Technology,}
\\
\small{ Wuhan 430065, P. R. China}}
\date{}
\maketitle

\begin{abstract}
In this paper, we study the existence and the concentration behavior of minimizers for $i_V(c)=\inf\limits_{u\in S_c}I_V(u)$, here $S_c=\{u\in H^1(\R^N)|~\int_{\R^N}V(x)|u|^2<+\infty,~|u|_2=c>0\}$ and
 $$I_V(u)=\frac{1}{2}\ds\int_{\R^N}(a|\nabla u|^2+V(x)|u|^2)+\frac{b}{4}\left(\ds\int_{\R^N}|\nabla u|^2\right)^2-\frac{1}{p}\ds\int_{\R^N}|u|^{p},$$
where $N=1,2,3$ and $a,b>0$ are constants. By the Gagliardo-Nirenberg inequality, we get the sharp existence of global constraint minimizers for $2<p<2^*$ when $V(x)\geq0$, $V(x)\in L^{\infty}_{loc}(\R^N)$ and $\lim\limits_{|x|\rightarrow+\infty}V(x)=+\infty$. For the case $p\in(2,\frac{2N+8}{N})\backslash\{4\}$, we prove the global constraint minimizers $u_c$ behave like
 $$
u_{c}(x)\approx \frac{c}{|Q_{p}|_2}\left(\frac{m_{c}}{c}\right)^{\frac{N}{2}}Q_p\left(\frac{m_{c}}{c}x-z_c\right).$$
for some $z_c\in\R^N$ when $c$ is large, where $Q_p$ is up to translations, the unique positive solution of $-\frac{N(p-2)}{4}\Delta Q_p+\frac{2N-p(N-2)}{4}Q_p=|Q_p|^{p-2}Q_p$ in $\R^N$ and $m_c=(\frac{\sqrt{a^2D_1^2-4bD_2i_0(c)}+aD_1}{2bD_2})^{\frac12}$,
$D_1=\frac{Np-2N-4}{2N(p-2)}$ and $D_2=\frac{2N+8-Np}{4N(p-2)}$.

\end{abstract}

\section{ Introduction and main result}
In this paper, we study the existence and the concentration phenomenon of normalized solutions to the following Kirchhoff equation
\begin{equation}\label{1.1}
    -\left(a+b\ds\int_{\R^N}|\nabla u|^2\right)\Delta u+(V(x)-\rho)u-|u|^{p-2}u=0,~~~x\in \R^N,~\rho\in\R,
\end{equation}
where $N\leq3$, $a,$ $b>0$ are constants and $2<p<2^*:=\frac{2N}{N-2},$ $2^*=6$ if $N=3$ and $2^*=+\infty$ if $N=1,2$. The potential  $V:\R^N\rightarrow\R$ is a suitable function.

In the past years, equation \eqref{1.1}, which is a nonlocal
one as the appearance of the term  $\int_{\R^N}|\nabla u|^2$, has attracted a lot of attention. \eqref{1.1} is no longer a pointwise identity, which causes some
mathematical difficulties and makes the study
particularly interesting. see \cite{ac,ds,ap1,b,cd,k,l,p} and the references therein. The first line to study \eqref{1.1} is to consider the case where $\rho$ is a fixed and assigned parameter. see e.g. \cite{hz,jw,ly1,ly2,lls,wt}. In such direction, the critical point theory is used to look for nontrivial solutions, however, nothing can be given a priori on the $L^2$-norm of the solutions. Recently, since the physicists are often interested in ``normalized solutions", solutions with prescribed $L^2$-norm are considered. A solution with $|u|_2=c$ corresponds to a critical point of the following $C^1$-functional
\begin{equation}\label{1.2}
I_{V}(u)=\ds\frac{1}{2}\ds\int_{\R^N}(a|\nabla u|^2+V(x)|u|^2)+\frac{b}{4}\left(\ds\int_{\R^N}|\nabla u|^2\right)^2
-\frac{1}{p}\ds\int_{\R^N}|u|^{p},~~~~u\in \mathcal{H},
\end{equation}
constrained on the $L^2$-sphere
$$S_c=\{u\in \mathcal{H}|~|u|_{2}=c>0\}$$
in $\mathcal{H}$, where $\mathcal{H}$ is defined as
$$\mathcal{H}=\left\{u\in H^1(\R^N)|~\int_{\R^N}V(x)|u|^2<+\infty\right\}$$
with the associated norm $\|u\|_{\mathcal{H}}=[\int_{\R^N}(|\nabla u|^2+|u|^2+V(x)|u|^2]^{\frac12}$.
Set
\begin{equation}\label{1.3}
i_{V}(c):=\inf\limits_{u\in S_c}I_V(u).
\end{equation}
We see that minimizers of $i_{V}(c)$ are critical points of $I_{V}|_{S_c}$ and the parameter $\rho$ is no longer fixed but appears as an associated Lagrange multiplier. If $V(x)\equiv0$, then the minimization problem \eqref{1.3} can be rewritten as
\begin{equation}\label{1.8}
i_0(c)=\inf\limits_{u\in \widetilde{S}_c}I_0(u),
\end{equation}
where $\widetilde{S}_c=\{u\in H^1(\R^N)|~|u|_2=c\}$ and
$$I_{0}(u)=\ds\frac{a}{2}\ds\int_{\R^N}|\nabla u|^2+\frac{b}{4}\left(\ds\int_{\R^N}|\nabla u|^2\right)^2
-\frac{1}{p}\ds\int_{\R^N}|u|^{p},~~~~~~~u\in H^1(\R^N).$$
When $b=0$, equation \eqref{1.1} does not depend on the nonlocal term any more, i.e. it becomes the following typical Schr\"{o}dinger equation (for simplicity, we may assume that $a\equiv1$)
\begin{equation}\label{1.4}
-\Delta u+(V(x)-\rho)u=|u|^{p-2}u,~~~x\in\R^N,~\rho\in\R,~N\geq1.
\end{equation}
In the literature, there are some papers studying the existence and the concentration behavior of normalized solutions to \eqref{1.4}, see e.g. \cite{c,j,m,s1,s2}. Similarly, a solution with $|u|_2=c$ can be obtained by looking for a minimizer of $e_{V}(c)=\inf\limits_{u\in S_c}E_V(u)$ (or $e_0(c)=\inf\limits_{u\in \widetilde{S}_c}E_0(u)$ for the case $V(x)\equiv0$), where
 $$
E_{V}(u)=\ds\frac{1}{2}\ds\int_{\R^N}(|\nabla u|^2+V(x)|u|^2)
-\frac{1}{p}\ds\int_{\R^N}|u|^{p},~~~~~~~u\in \mathcal{H}.
$$
It is showed that $p=\frac{2N+4}N$ is the $L^2$-critical exponent for $e_0(c)$, i.e. for all $c>0$, $e_0(c)>-\infty$ if $2<p<\frac{2N+4}{N}$ and $e_0(c)=-\infty$ if $\frac{2N+4}{N}<p<2^*$ (see e.g. \cite{c,j,s2}). Moreover, if $2<p<\frac{2N+4}{N}$, then $e_0(c)<0$ and $e_0(c)$ has a minimizer for all $c$; if $\frac{2N+4}{N}\leq p<2^*$, then $e_0(c)$ has no minimizers for all $c$. In \cite{m}, Maeda shows that up to translations, the minimizer of $e_0(c)$ is unique (see also \cite{gnn}\cite{k}) and gets a scaling property of the minimizer of $e_0(c)$ for $2<p<\frac{2N+4}{N}$, i.e. suppose that $u_c$ is a minimizer of $e_0(c),$ then up to translations,
\begin{equation}\label{1.12}
u_c(x)=c^{\frac{4}{2N+4-Np}}u_1(c^{\frac{2(p-2)}{2N+4-Np}}x)~~~~~\hbox{and}~~~~~e_0(c)=c^{\frac{4N-2p(N-2)}{2N+4-Np}}e_0(1).
\end{equation}
When $V(x)\not\equiv0$, Maeda in \cite{m} proves that $e_V(c)$ has a minimizer for all $c>0$ if $2<p<\frac{2N+4}{N}$ and $V(x)$ satisfies that $0=\inf\limits_{x\in\R^N}V(x)<\lim\limits_{|x|\rightarrow+\infty}V(x)=\sup\limits_{x\in\R^N}V(x)\leq \infty.$
In addition, if $V(x)\in C^1(\R^N)$ satisfies $|\nabla V(x)|\leq C(1+V(x))$ for some positive constant $C$, by using the essential scaling property \eqref{1.12} and a special $L^2$-preserving scaling, \cite{m} shows that there exists $y_c\in\R^N$ such that the minimizer of $e_V(c)$, denoted by $\tilde{u}_c$, satisfies $c^{-\frac{4}{2N+4-Np}}\tilde{u}_c(c^{-\frac{2(p-2)}{2N+4-Np}}x+y_c)\rightarrow u_1$ in $H^1(\R^N)$ as $c\rightarrow+\infty$, where $u_1$ is the minimizer of $e_0(1)$ given in \eqref{1.12}. For $p=\frac{2N+4}{N}$, when $V(x)$ is locally bounded and $V(x)\rightarrow+\infty$ as $|x|\rightarrow+\infty$, Guo et. al. in \cite{dgl}\cite{gs}\cite{gzz} proved that $e_{V}(c)$ has a minimizer if and only if $0<c<c_0$ for some $c_0>0$. Moreover, the minimizer of $e_{V}(c)$ concentrates at the flattest minimum of $V(x)$ as $c\nearrow c_0$.

 When $b>0$, due to the effect of the term $(\int_{\R^N}|\nabla u|^2)^2$ appeared in the energy functional, $p=\frac{2N+4}{N}$ is now $L^2$-subcritical for \eqref{1.3} and \eqref{1.8}. The existence of normalized solutions to equation \eqref{1.1} may be different from the Schr\"{o}dinger one. When $V(x)\equiv0$, by the $L^2$-preserving scaling and the well-known Gagliardo-Nirenberg inequality with the best constant \cite{we,kw}:
\begin{equation}\label{1.5}
\int_{\R^N}|u|^p\leq \frac{p}{2|Q_p|_2^{p-2}}|u|_2^{\frac{2N-p(N-2)}{2}}|\nabla u|_2^{\frac{N(p-2)}{2}},~~~~\forall~u\in H^1(\R^N),
\end{equation}
when either $p\in[2,\frac{2N}{N-2})$ with $N\geq3$ or $p\in[2,+\infty)$ with $N=1,2$, where the equality holds for $u=Q_p$ and $Q_p$ is, up to translations, the unique positive least energy solution of
\begin{equation}\label{1.6}
-\frac{N(p-2)}{4}\Delta Q_p+\frac{2N-p(N-2)}{4}Q_p=|Q_p|^{p-2}Q_p,~~~x\in\R^N,
\end{equation} it is showed in \cite{ye1,ye2} that $p=\frac{2N+8}{N}$ is the $L^2$-critical exponent for the minimization problem \eqref{1.8} in the sense that for each $c>0$, $i_0(c)>-\infty$ if $2<p<\frac{2N+8}{N}$ and $i_0(c)=-\infty$ if $\frac{2N+8}{N}<p<2^*$. Moreover, we collect some known results in the following proposition concerning the minimizers of $i_0(c)$.

\begin{proposition}\label{lem1.1}(\cite{ye1}, Theorem 1.1)~~Assume that $2<p<2^*$ and $N\leq3$.

 (1)~~If $2<p<\frac{2N+8}{N}$, then

~~~~~(i)~~there exists $$ \left\{
 \begin{array}{ll}
 c_*=0,\,\,\, &~2<p<\frac{2N+4}{N}, \vspace{0.2cm}\\
 c_*=a^{\frac{N}{4}}|Q_{\frac{2N+4}{N}}|_{2},\,\,\, &~p=\frac{2N+4}{N}, \vspace{0.2cm}\\
 c_*\in(0,+\infty),\,\,\, &~\frac{2N+4}{N}< p<\frac{2N+8}{N},
 \end{array}
 \right.$$
  such that
$i_0(c)=0~~\hbox{for~all}~0<c\leq c_*$ and $i_0(c)<0~~\hbox{for~all}~c>c_*.
$

~~~~(ii)~~$i_0(c)$ has a minimizer if and only if $c\in T$, where
\begin{equation}\label{1.13} T=\left\{
 \begin{array}{ll}
 (0,+\infty),\,\,\, &~2<p<\frac{2N+4}{N}, \vspace{0.2cm}\\
 (a^{\frac{N}{4}}|Q_{\frac{2N+4}{N}}|_{2},+\infty),\,\,\, &~p=\frac{2N+4}{N}, \vspace{0.2cm}\\

 [ c_*,+\infty),\,\,\, &~\frac{2N+4}{N}< p<\frac{2N+8}{N}.
 \end{array}
 \right.
 \end{equation}

(2)~~If $p=\frac{2N+8}{N},$ then $i_0(c)=0$ for all $ 0<c\leq(\frac{b}{2}|Q_{\frac{2N+8}{N}}|_2^{\frac{8}{N}})^{\frac{N}{8-2N}}$ and $i_0(c)=-\infty$ for all $ c>(\frac{b}{2}|Q_{\frac{2N+8}{N}}|_2^{\frac{8}{N}})^{\frac{N}{8-2N}}.$
Moreover, $i_0(c)$ has no minimizer for all $c>0$.

(3)~~If $\frac{2N+8}{N}<p<2^*$, then $i_0(c)=-\infty$ and $i_0(c)$ has no minimizer for all $c>0$.
\end{proposition}
If $V(x)\not\equiv0$, then the case becomes totally different from the autonomous one since the $L^2$-preserving scaling argument fails here. As far as we know, there is no paper on this respect. In this paper, we study the existence and the concentration behavior of minimizers for $i_V(c)$ when $V(x)$ satisfies the following condition:
$$~~~~~~~~~~~~~V(x)\in L_{loc}^{\infty}(\R^N),~~~\inf\limits_{x\in\R^N}V(x)=0~~~
\hbox{and}~~~\lim_{|x|\rightarrow+\infty}V(x)=+\infty.~~~~~~~~~~~~~~(V)$$

We call $(u_c,\rho_c)\in S_c\times\R$ a couple of solution to the equation \eqref{1.1} if $(u_c,\rho_c)$ satisfies \eqref{1.1} and $|u_c|_2=c$.

Our first result is as follows.

\begin{theorem}\label{th1.1}~~Suppose that $V(x)$ satisfies $(V)$, $2<p<2^*$ and $N\leq3$.

(1)~~For each
$$\left\{
 \begin{array}{ll}
 c>0,\,\,\, &~2<p<\frac{2N+8}{N}, \vspace{0.2cm}\\
 0<c\leq (\frac{b}{2}|Q_{\frac{2N+8}{N}}|_2^{\frac{8}{N}})^{\frac{N}{8-2N}},\,\,\, &~p=\frac{2N+8}{N},
 \end{array}
 \right.$$
  $i_{V}(c)$ has at least one minimizer. Moreover, there exists a couple of solution $(u_c,\rho_c)\in S_c\times\R$  to \eqref{1.1} with $I_{V}(u_c)=i_{V}(c).$

(2)~~For each $c>(\frac{b}2|Q_{\frac{2N+8}{N}}|_2^{\frac{8}{N}})^{\frac{N}{8-2N}}$ if $p=\frac{2N+8}{N}$ or $c>0$ if $\frac{2N+8}{N}<p<2^*$,
 $i_V(c)=-\infty$ and $i_V(c)$ has no minimizer.
\end{theorem}

Note that $i_{V}(c)\geq i_{0}(c)$ for all $c>0$. Under condition $(V)$, the embedding $\mathcal{H}\hookrightarrow L^q(\R^N),$ $q\in[2,2^*)$ is compact (see e.g. \cite{bw}). To prove Theorem \ref{th1.1} (1), it is enough to show the boundedness of each minimizing sequence for $i_{V}(c)$, which can be obtained by using the inequality \eqref{1.5}. By using the least energy solution $Q_p$ as a test function and a suitable $L^2$-preserving scaling, we succeeded in obtaining Theorem \ref{th1.1} (2).

Our main result concerns the concentration phenomena of minimizers $u_c$ of $i_{V}(c)$ as $c\rightarrow+\infty$ for $2<p<\frac{2N+8}{N}$. Since $\lim\limits_{c\rightarrow+\infty}i_0(c)=-\infty$ (see Lemma \ref{lem4.4} below), it is reasonable to conjecture that $i_V(c)\rightarrow-\infty$ as $c\rightarrow+\infty$. To do so, we need to study the property of minimizers for $i_0(c)$. The theorem below summarizes some properties of $i_0(c)$.

\begin{theorem}\label{th1.2}~~Let $2<p<\frac{2N+8}{N}$ and $c\in T$. Suppose that $v_c\in \widetilde{S}_c$ is a minimizer of $i_0(c)$.

(1)~~Up to translations,
$$v_c(x)=\frac{c}{|Q_p|_2}(\frac{m_c}{c})^{\frac N2}Q_p(\frac{m_c}{c}x),$$
where $m_c=(\frac{\sqrt{a^2D_1^2-4bD_2i_0(c)}+aD_1}{2bD_2})^{\frac12}$,
$D_1=\frac{N(p-2)-4}{2N(p-2)}$ and $D_2=\frac{2N+8-Np}{4N(p-2)}.$ That is to say, $i_0(c)$ has a unique minimizer for each $c\in T$. Moreover,
$$i_0(c)=aD_1m_c^2-bD_2m_c^4.$$

(2)~~ $\lim\limits_{c\rightarrow+\infty}\frac{i_0(c)}{c^2}=-\infty$.

(3)~~ If $p\neq4$, then either $\frac{m_c}{c}\rightarrow+\infty$ or $\frac{m_c}{c}\rightarrow 0$ as $c\rightarrow+\infty$.

~~~~~~If $p=4$, then $\frac{m_c}{c}\rightarrow(\frac{N}{2b|Q_4|_2^2})^{\frac{1}{4-N}}$ as $c\rightarrow+\infty$.

$(4)$~~There exists $\mu_c=-\frac{2N-p(N-2)}{4}\frac{c^{p-2-\frac{N(p-2)}{2}}}{|Q_p|_2^{p-2}}m_c^{\frac{N(p-2)}{2}}<0$  such that $(v_c,\mu_c)$ satisfies the following equation
$$-\left(a+b\ds\int_{\R^N}|\nabla v_c|^2\right)\Delta v_c-|v_c|^{p-2}v_c=\mu_c v_c,~~~x\in \R^N.$$
\end{theorem}
Then our main result is as follows.

\begin{theorem}\label{th1.3}~~Suppose that $V(x)$ satisfies $(V)$, $N\leq3$ and $p\in(2,\frac{2N+8}{N})\backslash\{4\}$. For any sequence $\{c_n\}\subset(0,+\infty)$ with $c_n\rightarrow+\infty$ as $n\rightarrow+\infty$, let $(u_{c_n},\rho_{c_n})\in S_{c_n}\times\R$ be the couple of solution to \eqref{1.1} obtained in Theorem \ref{th1.1}. Then there exists a subsequence of $\{c_n\}$ (still denoted by $\{c_n\}$) and a sequence $\{z_n\}\subset\R^N$ such that
$$\frac{|Q_{p}|_2}{c_n}\left(\frac{c_n}{m_{c_n}}\right)^{\frac{N}{2}}
u_{c_n}\left(\frac{c_n}{m_{c_n}}\left(x+z_n\right)\right)\longrightarrow Q_p(x)$$
 in $L^q(\R^N)$ for all $2\leq q<2^*$ and
 $$\left[\frac{1}{c_n}(\frac{c_n}{m_{c_n}})^{\frac{N}{2}}\right]^{p-2}\rho_{c_n}\rightarrow-\frac{4|Q_p|_2^{p-2}}{2N-p(N-2)}$$ as $n\rightarrow+\infty.$
\end{theorem}

\begin{remark}Although the equation \eqref{1.1} is ``nonlocal", the concentration behavior of the solution to \eqref{1.1} given in Theorem \ref{th1.3} when the $L^2$-norm of the solution is ``large", is similar to that of the solution to \eqref{1.4}, which is ``local" (see Lemma 4.6 in \cite{m}). However, our argument is different.
\end{remark}

We give the main idea of the proof of Theorem \ref{th1.3}. As the problem we deal with is nonlocal, the approach of \cite{m} considering mass concentration of Schr\"{o}dinger equation with the $L^2$-subcritical exponent cannot be applied directly for two reasons. First, since the terms in $I_0(u)$ or in $I_V(u)$ scale differently in space, one cannot hope to get rid of the nonlocal term by scaling in space. Then we cannot follow the argument in \cite{m} to show that the minimizer of $i_0(c)$ possesses a property like \eqref{1.12}, which makes it difficult to study the convergence of $i_V(c)$ as $c\rightarrow+\infty$. To overcome this difficulty, \textbf{the key point is to prove the uniqueness of the minimizer $v_c$ of $i_0(c)$ and to give the accurate expression of $v_c$ (see Theorem \ref{th1.2}).} This is not easy and it needs much more analysis. Based on this property, we show that
\begin{equation}\label{1.7}
\begin{array}{ll}
&\ds\lim\limits_{n\rightarrow+\infty}\frac{\int_{\R^N}V(x)[\varphi(x-x_0)v_{c_n}(x-x_0)]^2}{c_n^2}\\[5mm]
&~~~~~~~~~~~~~~~~~\ds=\lim\limits_{n\rightarrow+\infty}(\frac{m_{c_n}}{c_n})^N\int_{\R^N}V(x+x_0)[\varphi(x)Q_p(\frac{m_{c_n}}{c_n}x)]^2=0,
\end{array}
\end{equation}
where $x_0\in\R^N$  and $\varphi\in C_0^{\infty}(\R^N)$ is a cut-off function. Then $\frac{i_V(c_n)}{c_n^2}-\frac{i_0(c_n)}{c_n^2}\rightarrow0~
\hbox{as}~n\rightarrow+\infty.$ By Theorem \ref{th1.2} (2) we get an essential estimate for $i_V(c_n)$, i.e.
\begin{equation}\label{1.11}
\lim\limits_{n\rightarrow+\infty}\frac{i_{V}(c_n)}{i_0(c_n)}=1.
 \end{equation}
 Indeed, \textbf{to show \eqref{1.7}, it requires that $\frac{m_{c_n}}{c_n}$ converges to either 0 or $+\infty$, which is the reason why $p\neq4$ is assumed in Theorem \ref{th1.3}} (see Theorem \ref{th1.2} (3)).
 Second, due to the effect of the nonlocal term, property \eqref{1.12} no longer holds for $i_0(c)$ and the minimizer of $i_0(c)$ (see Theorem \ref{th1.2} (1) above for details), which makes that the special $L^2$-scaling using in \cite{m} to get the mass concentration cannot be used here. We overcome this difficulty by using the estimate \eqref{1.11} to obtain the following optimal energy estimates for each minimizer $u_{c_n}$ of $i_V(c_n)$:
\begin{equation}\label{1.14}
\frac{\int_{\R^N}|\nabla u_{c_n}|^2}{m_{c_n}}\rightarrow1,~~~~~~~~~~~~~~~~~~
\frac{\int_{\R^N}|u_{c_n}|^{p}}{i_0(c_n)}\rightarrow-\frac{8p}{2N+8-Np}
\end{equation}
and
\begin{equation}\label{1.15}
\frac{\int_{\R^N}V(x)|u_{c_n}|^2}{i_0(c_n)}\rightarrow0
\end{equation}
as $n\rightarrow+\infty$, i.e. $\frac{I_0(u_{c_n})}{i_0(c_n)}\rightarrow1$, which gives us a cue that $u_{c_n}$ might behave like the minimizer of $i_0(c_n)$. It is necessary to point out that the nonlocal term plays an important role in the proof of \eqref{1.14} and \eqref{1.15}. Finally, for any $c\in T$, we set $w_{n}(x):=\frac{c}{c_n}u_{c_n}^{\frac{c_n m_c}{m_{c_n} c}}(x),$ by \eqref{1.14} \eqref{1.15} we see that $\{w_{n}\}$ is a bounded minimizing sequence of $i_{0}(c)$ and then $w_n$ converges strongly to the minimizer of $i_0(c)$ in $H^1(\R^N)$. So the theorem is proved.

Throughout this paper, we use standard notations. For simplicity, we
write $\int_{\Omega} h$ to mean the Lebesgue integral of $h(x)$ over
a domain $\Omega\subset\R^N$. $L^{p}:= L^{p}(\R^{N})~(1\leq
p\leq+\infty)$ is the usual Lebesgue space with the standard norm
$|\cdot|_{p}.$ We use `` $\rightarrow"$ and `` $\rightharpoonup"$ to denote the
strong and weak convergence in the related function space
respectively. $C$ will
denote a positive constant unless specified. We use `` $:="$ to denote definitions and $B_r(x):=\{y\in\R^N|\,|x-y|<r\}$. We denote a subsequence
of a sequence $\{u_n\}$ as $\{u_n\}$ to simplify the notation unless
specified.

The paper is organized as follows. In $\S$ 2, we prove Theorem \ref{th1.2}. In $\S$ 3, we prove Theorem \ref{th1.1} and Theorem \ref{th1.3}.

\section{Proof of Theorem \ref{th1.2}}
In this section, we will present an accurate description of $ i_0(c)=\inf\limits_{u\in\widetilde{ S}_c}I_0(u),$
where
$$
I_{0}(u)=\ds\frac{a}{2}\ds\int_{\R^N}|\nabla u|^2+\frac{b}{4}\left(\ds\int_{\R^N}|\nabla u|^2\right)^2
-\frac{1}{p}\ds\int_{\R^N}|u|^{p},~~~~u\in H^1(\R^N),
$$
$2<p<\frac{2N+8}N$ and $N\leq 3$. Recall that $i_0(c)$ has a minimizer if and only if $c\in T$, where $T$ is defined in \eqref{1.13}.

For simplicity, in what follows we denote
\begin{equation}\label{4.1}
D_1:=\frac{N(p-2)-4}{2N(p-2)},~~~~~~D_2:=\frac{2N+8-Np}{4N(p-2)}.
\end{equation}
Note that $D_2>0$ for each $2<p<\frac{2N+8}N$.

For any $u\in H^1(\R^N)$ and any $t>0$, in what follows we denote
\begin{equation}\label{2.1}
u^t(x):=t^{\frac N2}u(tx).
\end{equation}
Then $u^t\in \widetilde{S}_c$ if $u\in \widetilde{S}_c$.

By \eqref{1.5}\eqref{1.6} and the corresponding Pohozaev identity we see that
\begin{equation}\label{2.3}
\int_{\R^N}|\nabla Q_p|^2=\frac{2}{p}\int_{\R^N}|Q_p|^{p}=\int_{\R^N}|Q_p|^2.
\end{equation}
It is proved in \cite{gnn} that $Q_p$ is decreasing away from the origin and
$$Q_p(x),|\nabla Q_p(x)|=O(|x|^{-\frac12}e^{-|x|})~~~~~\hbox{as}~|x|\rightarrow+\infty.$$

The following lemma, which will be useful in the main proof, can be easily obtained, so we omit the proof.

\begin{lemma}\label{lem4.1}
Suppose that $v_c\in \widetilde{S}_c$ is a minimizer of $i_0(c)$, then $v_c$ is positive.
\end{lemma}

\begin{lemma}\label{lem4.2} Let $2<p<\frac{2N+8}N$ and $c\in T$, suppose that $v_c\in \widetilde{S}_c$ is a minimizer of $i_0(c)$, then up to translations,
$$v_c(x)=\frac{c}{|Q_p|_2}Q_p^{\frac{m_c}{c}}(x),$$
where $m_c=(\frac{\sqrt{a^2D_1^2-4bD_2i_0(c)}+aD_1}{2bD_2})^{\frac12}$.
\end{lemma}
\begin{proof}~~By the Lagrange multiplier theorem, there exists $\mu_c\in\R$ such that $(v_c,\mu_c)$ satisfies the following equation
\begin{equation}\label{4.2}
-\left(a+b\int_{\R^N}|\nabla v_c|^2\right)\Delta v_c-|v_c|^{p-2}v_c=\mu_c v_c,~~~x\in\R^N.
\end{equation}
By Lemma 2.1 in \cite{ly1}, we see that $v_c$ satisfies the following Pohozaev identity
$$\frac{N-2}{2}\left[a\int_{\R^N}|\nabla v_c|^2+b\left(\int_{\R^N}|\nabla v_c|^2\right)^2\right]-\frac{N}{2}\mu_cc^2-\frac{N}{p}\int_{\R^N}|v_c|^{p}=0.$$
Hence
\begin{equation}\label{4.3}
a\int_{\R^N}|\nabla v_c|^2+b\left(\int_{\R^N}|\nabla v_c|^2\right)^2=\frac{N(p-2)}{2p}\int_{\R^N}|v_c|^{p}.
\end{equation}
So
\begin{equation}\label{4.4}
\mu_c=-\frac{2N-p(N-2)}{N(p-2)c^2}\left(a+b\int_{\R^N}|\nabla v_c|^2\right)\int_{\R^N}|\nabla v_c|^2
\end{equation}
and \begin{equation}\label{4.11}
i_0(c)=I_0(v_c)=aD_1\int_{\R^N}|\nabla v_c|^2-bD_2\left(\int_{\R^N}|\nabla v_c|^2\right)^2.
\end{equation}
 Since $D_2>0$ and $i_0(c)\leq0$, we have
\begin{equation}\label{4.5}
\int_{\R^N}|\nabla v_c|^2=\frac{\sqrt{a^2D_1^2-4bD_2i_0(c)}+aD_1}{2bD_2}=m_c^2.
\end{equation}
Therefore, by Lemma \ref{lem4.1} and \eqref{4.2}-\eqref{4.5} we see that $v_c\in \widetilde{S}_c$ is a positive solution of the following equation
$$-\frac{N(p-2)}4\Delta v_c+\frac{2N-p(N-2)}{4}(\frac{m_c}{c})^2 v_c=\frac{N(p-2)}{4(a+bm_c^2)}|v_c|^{p-2}v_c,~~~x\in\R^N.
$$
Set  $v_c(x):=\left[\frac{4a+4bm_c^2}{N(p-2)}\right]^{\frac{1}{p-2}}
\left(\frac{c}{m_c}\right)^{\frac{N(p-2)-4}{2(p-2)}}w^{\frac{m_c}{c}}(x),$ then $w$ satisfies the equation
\begin{equation}\label{4.10}-\frac{N(p-2)}4\Delta w+\frac{2N-p(N-2)}{4}w=|w|^{p-2}w,~~~x\in\R^N.
\end{equation}
By the fact that $w$ is positive together with the uniqueness of positive solutions (up to translations) to the equation \eqref{4.10}, we conclude that $w=Q_p$, i.e. $$v_c(x)=\left[\frac{4a+4bm_c^2}{N(p-2)}\right]^{\frac{1}{p-2}}
\left(\frac{c}{m_c}\right)^{\frac{N(p-2)-4}{2(p-2)}}Q_p^{\frac{m_c}{c}}(x).$$
By direct calculation, we see that
$$c^2=\int_{\R^N}|v_c|^2=\left[\frac{4a+4bm_c^2}{N(p-2)}\right]^{\frac{2}{p-2}}
\left(\frac{c}{m_c}\right)^{\frac{N(p-2)-4}{p-2}}|Q_p|_2^2,$$
then
\begin{equation}\label{4.12}
[\frac{4a+4bm_c^2}{N(p-2)}]^{\frac{1}{p-2}}(\frac{c}{m_c})^{\frac{N(p-2)-4}{2(p-2)}}=\frac{c}{|Q_p|_2}.
\end{equation}
Therefore, up to translations, $v_c=\frac{c}{|Q_p|_2}Q_p^{\frac{m_c}{c}}$ is the unique minimizer of $i_0(c).$
\end{proof}

\begin{corollary}\label{cor4.3}

(1)~~We conclude from the proof of Lemma \ref{lem4.2} that
$$i_0(c)=aD_1m_c^2-bD_2m_c^4$$ and
\begin{equation}\label{4.6}
a+bm_c^2=\frac{N(p-2)}{4}\frac{c^{\frac{2N-p(N-2)}{2}}}{|Q_p|_2^{p-2}}
m_c^{\frac{N(p-2)-4}{2}}.
\end{equation}

(2)~~There exists $\mu_c=-\frac{2N-p(N-2)}{4}\frac{c^{p-2-\frac{N(p-2)}{2}}}{|Q_p|_2^{p-2}}
m_c^{\frac{N(p-2)}{2}}<0$ such that $(\frac{c}{|Q_p|_2}Q_p^{\frac{m_c}{c}},\mu_c)$ is a couple of solution to the equation \eqref{4.2}.

(3)~~In particular, if $p=\frac{2N+4}N$, then we can give accurate expressions of $i_0(c)$ and $m_c$, i.e. $$i_0(c)=-\frac{1}{4b}\Big{[}(\frac{c}{|Q_{\frac{2N+4}{N}}|_2})^{\frac 4N}-a\Big{]}^2,$$ $$m_c=b^{-\frac12}\Big{[}(\frac{c}{|Q_{\frac{2N+4}{N}}|_2})^{\frac 4N}-a\Big{]}^{\frac12}$$
and
$$\mu_c=-\frac{2}{Nbc^2}(\frac{c}{|Q_{\frac{2N+4}{N}}|_2})^{\frac 4N}\Big{[}(\frac{c}{|Q_{\frac{2N+4}{N}}|_2})^{\frac 4N}-a\Big{]}.$$
\end{corollary}
\begin{proof}~~
(1)(2) follow from \eqref{4.2}, \eqref{4.4}-\eqref{4.5} and \eqref{4.12}.

(3)~~When $p=\frac{2N+4}{N}$, $D_1=0$ and $D_2=\frac14.$ By \eqref{4.6} we see that $m_c^2=\frac1b[(\frac{c}{|Q_{\frac{2N+4}{N}}|_2})^{\frac4N}-a]$. So by (1)(2) we have
$i_0(c)=-\frac{b}4m_c^4=-\frac{1}{4b}[(\frac{c}{|Q_{\frac{2N+4}{N}|_2}})^{\frac4N}-a]^2$ and $$\mu_c=-\frac2{Nc^2}(\frac{c}{|Q_{\frac{2N+4}{N}}|_2})^{\frac4N}m_c^2=-\frac{2}{Nbc^2}(\frac{c}{|Q_{\frac{2N+4}{N}}|_2})^{\frac 4N}\Big{[}(\frac{c}{|Q_{\frac{2N+4}{N}}|_2})^{\frac 4N}-a\Big{]}.$$
\end{proof}
The following proposition follows from Theorem 1.1 in \cite{ye1} and Lemma \ref{lem4.2}.
\begin{proposition}\label{pro4.6}~~Let $2<p<\frac{2N+8}{N}$ and $c\in T$. Suppose that $\{u_n\}\subset\widetilde{S}_c$ is a minimizing sequence of $i_0(c)$, then there exists a subsequence of $\{u_n\}$ (still denoted by $\{u_n\}$), and $\{y_n\}\subset\R^N$ such that $$u_n(x+y_n)\rightarrow \frac{c}{|Q_p|_2}Q_p^{\frac{m_c}{c}}(x)~~~~~\hbox{in}~H^1(\R^N).$$
\end{proposition}

Based on Corollary \ref{cor4.3}, we could get the exact value of $c_*$ defined as in Proposition \ref{lem1.1} when $\frac{2N+4}{N}<p<\frac{2N+8}{N},$ which is not obtained in \cite{ye1}.
\begin{lemma}\label{lem4.7}~~If $\frac{2N+4}{N}<p<\frac{2N+8}{N},$ then $c_*=[\frac{4a|Q_p|_2^{p-2}}{2N+8-Np}\left(\frac{2a}{b}\frac{N(p-2)-4}{2N+8-Np}\right)^{\frac{4-N(p-2)}4}]^{\frac{2}{2N-p(N-2)}}$.
\end{lemma}
\begin{proof}~~By Proposition \ref{lem1.1} (1), when $\frac{2N+4}{N}<p<\frac{2N+8}{N},$ $i_0(c_*)=0$ and $i_0(c_*)$ has a minimizer. Then by Corollary \ref{cor4.3} (1) we see that $m_{c_*}=\sqrt{\frac{aD_1}{bD_2}}$. Hence we conclude from \eqref{4.6} and the definitions of $D_1,D_2$ that
$$c_*=\Big{[}\frac{4a|Q_p|_2^{p-2}}{2N+8-Np}\left(\frac{2a}{b}\frac{N(p-2)-4}{2N+8-Np}\right)^{\frac{4-N(p-2)}4}\Big{]}^{\frac{2}{2N-p(N-2)}}.$$
\end{proof}

Note that $2<4<\frac{2N+8}{N}$ when $N\leq3.$

\begin{lemma}\label{lem4.5}~~Let $2<p<\frac{2N+8}{N}$ and $c>c_*$, where $c_*$ is given in Proposition \ref{lem1.1} (1). Then

(1)~~$\lim\limits_{c\rightarrow+\infty}\frac{i_0(c)}{c^2}=-\infty$.

(2)~~If $p\neq4$, then either $\frac{m_c}{c}\rightarrow+\infty$ or $\frac{m_c}{c}\rightarrow0$ as $c\rightarrow+\infty$.

~~~~~~If $p=4$, then $\frac{m_c}{c}\rightarrow(\frac{N}{2b|Q_4|_2^2})^{\frac{1}{4-N}}$ as $c\rightarrow+\infty$.
\end{lemma}
\begin{proof}
(1)~~It is proved in Lemma 2.5 of \cite{ye1} that $c\mapsto\frac{i_0(c)}{c^2}$ is strictly decreasing on $(c_*,+\infty)$. For readers' convenience, we give a detailed proof of this fact here.

For any $c_1,c_2\in (c_*,+\infty)$ with $c_1<c_2$, let $\{u_n\}\subset \widetilde{S}_{c_1}$ be a minimizing sequence for $i_0(c_1)$. By Proposition \ref{lem1.1} (1), we have $i_0(c_1)<0$. Then there exist $0<k_1<k_2$ independent of $n$ such that
$$k_1\leq\int_{\R^N}|\nabla u_n|^2\leq k_2.$$
Set $\theta:=\frac{c_2}{c_1}>1$ and $u_{n,\theta}(x):=u_n(\theta^{-\frac{2}{N}}x).$ Then $u_{n,\theta}\in \widetilde{S}_{\theta c_1}$ and we have that
$$\begin{array}{ll}
I_0(u_{n,\theta})&=\ds\theta^2I_0(u_n)+\theta^2\left[(\theta^{-\frac{4}{N}}-1) \frac{a\int_{\R^N}|\nabla u_n|^2}{2}+(\theta^{2-\frac{8}{N}}-1)\frac{b(\int_{\R^N}|\nabla u_n|^2)^2}4\right]\\[5mm]
&\leq\ds\theta^2I_0(u_n)-\ds\theta^2\left[(1-\theta^{-\frac{4}{N}})\frac{ak_1}{2}+(1-\theta^{2-\frac{8}{N}})\frac{bk_1^2}{4}\right].
\end{array}$$
Let $n\rightarrow+\infty$ and notice that the second term of r.h.s. above is strictly negative and independent of $n$, it follows that $i_{0}(\theta c_1)<\theta^2 i_{0}(c_1),$ i.e.
$$
\frac{i_0(c_2)}{c_2^2}<\frac{ i_0(c_1)}{c_1^2}.
$$
So the function $c\mapsto\frac{i_0(c)}{c^2}$ is strictly decreasing on $(c_*,+\infty)$.

To prove the lemma, since $\frac{i_0(c)}{c^2}<0$ for all $c>c_*$, we see that either $\frac{i_0(c)}{c^2}\rightarrow-\infty$ or $\frac{i_0(c)}{c^2}\rightarrow-A_0$ for some $A_0>0$ as $c\rightarrow+\infty.$ If $\frac{i_0(c)}{c^2}\rightarrow-A_0$ for some $A_0>0$, then by the definition of $m_c$, we see that $\frac{m_c}{c^{\frac12}}\rightarrow A_1$ for some constant $A_1>0$. By \eqref{4.6} we have
\begin{equation}\label{4.8}
\begin{array}{ll}
\ds\frac{a+bm_c^2}{c}=\frac{N(p-2)}{4|Q_p|_2^{p-2}}c^{\frac{(4-N)(p-2)}4}(\frac{m_c}{c^\frac12})^{\frac{N(p-2)-4}{2}}.
\end{array}
\end{equation}
Since $N\leq 3$ and $p>2$, letting $c\rightarrow+\infty$ in \eqref{4.8} we get that $bA_1^2=+\infty$, which is impossible. So $\frac{i_0(c)}{c^2}\rightarrow-\infty$ as $c\rightarrow+\infty.$

(2)~~Similarly to \eqref{4.8}, by \eqref{4.6} we see that $\frac{a+bm_c^2}{c^2}=\frac{N(p-2)}{4|Q_p|_2^{p-2}}c^{p-4}(\frac{m_c}{c})^{\frac{N(p-2)-4}{2}},$ i.e.
\begin{equation}\label{4.9}
\begin{array}{ll}
\ds\frac{a}{c^2}(\frac{m_c}{c})^{\frac{2N+4-Np}{2}}+b(\frac{m_c}{c})^{\frac{2N+8-Np}{2}}=\frac{N(p-2)}{4|Q_p|_2^{p-2}}c^{p-4}.
\end{array}
\end{equation}
Let us consider the following two cases.

Case 1: $p\neq4$.

If $2<p<4$, then $\frac{2N+8-Np}{2}>0$ and it follows from \eqref{4.9} that
$$b(\frac{m_c}{c})^{\frac{2N+8-Np}{2}}\leq\frac{N(p-2)}{4|Q_p|_2^{p-2}}c^{p-4}\rightarrow0~~~\hbox{as}~c\rightarrow+\infty,$$
which implies that $\lim\limits_{c\rightarrow+\infty}\frac{m_c}{c}=0$. If $4<p<\frac{2N+8}{N}$, then we conclude from \eqref{4.9} again that either $\frac{a}{c^2}(\frac{m_c}{c})^{\frac{2N+4-Np}{2}}\rightarrow+\infty$ or $b(\frac{m_c}{c})^{\frac{2N+8-Np}{2}}\rightarrow+\infty$ as $c\rightarrow+\infty$, which both imply either $\lim\limits_{c\rightarrow+\infty}\frac{m_c}{c}=+\infty$ or $\lim\limits_{c\rightarrow+\infty}\frac{m_c}{c}=0$.

Case 2: $p=4$.

If $p=4$, then \eqref{4.9} is simplified to be
$$ \frac{N}{2|Q_4|_2^{2}}=\ds\frac{a}{c^2}(\frac{m_c}{c})^{2-N}+b(\frac{m_c}{c})^{4-N}=\left(\frac{a}{m_c^2}+b\right)(\frac{m_c}{c})^{4-N},
$$
i.e. $$(\frac{m_c}{c})^{4-N}=\frac{N}{2|Q_4|_2^{2}(\frac{a}{c}\frac{c}{m_c^2}+b)}.$$
By (1) and the definition of $m_c$, we see that $\lim\limits_{c\rightarrow+\infty}\frac{m_c^2}{c}=+\infty$. Then $\frac{m_c}{c}\rightarrow(\frac{N}{2b|Q_4|_2^2})^{\frac{1}{4-N}}$ as $c\rightarrow+\infty$. Therefore the lemma is proved.
\end{proof}

\begin{lemma}\label{lem4.4}~~Let $2<p<\frac{2N+8}{N},$ then $\left\{
 \begin{array}{ll}
  i_0(c)\rightarrow-\infty,\,\,\, &  \vspace{0.2cm}\\
 m_c\rightarrow+\infty,\,\,\, &  \vspace{0.2cm}\\
  \mu_c\rightarrow-\infty,\,\,\, &
 \end{array}
 \right.$  as $c\rightarrow+\infty$.
\end{lemma}
\begin{proof}~~By Lemma \ref{lem4.5} (1), one easily sees that $\lim\limits_{c\rightarrow+\infty}i_0(c)=-\infty$. Then by Corollary \ref{cor4.3} (1) we have $m_c\rightarrow+\infty$ as $c\rightarrow+\infty$. Hence by Corollary \ref{cor4.3} (1)(2) and \eqref{4.6} we see that
$$\begin{array}{ll}
\ds\mu_c&=\ds-\frac{2N-p(N-2)}{N(p-2)}\frac{am_c^2+bm_c^4}{c^2}\\[5mm]
&=\ds\frac{2N-p(N-2)}{N(p-2)}\frac{a+bm_c^2}{bD_2m_c^2-aD_1}\frac{i_0(c)}{c^2}\rightarrow-\infty
\end{array}$$
as $c\rightarrow+\infty$.
\end{proof}

\noindent $\textbf{Proof of Theorem \ref{th1.2}}$\,\,\

\begin{proof}~~Theorem \ref{th1.2} follows directly from Lemmas \ref{lem4.2}, \ref{lem4.5} and Corollary \ref{cor4.3}.

\end{proof}

\section{Proof of Theorems \ref{th1.1} and \ref{th1.3}}

In this section, we first consider the minimization problem \eqref{1.3} when $V(x)$ satisfies condition $(V)$. We need the following compactness result, see e.g. \cite{bw}.

\begin{lemma}\label{lem3.1}~~
Suppose that $V(x)\in L^{\infty}_{loc}(\R^N)$ with $\lim\limits_{|x|\rightarrow +\infty}V(x)=+\infty$. Then the embedding $\mathcal{H}\hookrightarrow L^q(\R^N),~2\leq q<2^*$ is compact.
\end{lemma}

\noindent $\textbf{Proof of Theorem \ref{th1.1}}$\,\,\

\begin{proof}~~(1)~~Since $V(x)\geq0$, we have $I_V(u)\geq I_0(u)$ for any $u\in \mathcal{H}.$ Then $i_V(c)\geq i_0(c)$ for each $c\in \widetilde{T},$ where
$$
\widetilde{T}=\left\{
 \begin{array}{ll}
 (0,+\infty),\,\,\, &~2<p<\frac{2N+8}{N}, \vspace{0.2cm}\\
 (0,(\frac{b}{2}|Q_{\frac{2N+8}{N}}|_2^{\frac{8}{N}})^{\frac{N}{8-2N}}],\,\,\, &~p=\frac{2N+8}{N}.
 \end{array}
 \right.
$$
 So we see from Proposition \ref{lem1.1} that $i_V(c)$ is well defined for $c\in \widetilde{T}$.

For any $c\in \widetilde{T}$, let $\{u_n\}\subset S_c$ be a minimizing sequence for $i_V(c)$, i.e. $I_V(u_n)\rightarrow i_V(c)$ as $n\rightarrow+\infty$. By Lemma \ref{lem3.1}, to prove that $i_V(c)$ has a minimizer it is enough to prove that $\{u_n\}$ is uniformly bounded in $\mathcal{H}$.

Indeed, if $\{u_n\}$ is uniformly bounded in $\mathcal{H}$, then up to a subsequence, there exists $u_c\in \mathcal{H}$ such that $u_n\rightharpoonup u_c$ in $\mathcal{H}$. Lemma \ref{lem3.1} shows that $u_n\rightarrow u_c$ in $L^q(\R^N)$, $2\leq q<2^*$, which implies that $|u_c|_2=c$, i.e. $u_c\in S_c$. By the weak lower semicontinuity of the norm in $\mathcal{H}$, we have $i_V(c)\leq I_V(u_c)\leq \liminf\limits_{n\rightarrow+\infty}I_V(u_n)=i_V(c)$, i.e. $u_c$ is a minimizer for $i_V(c)$. So $i_V(c)$ has at least one minimizer for each $c\in \widetilde{T}$.

 Let us next prove the boundedness of $\{u_n\}$. If $2<p<\frac{2N+8}{N}$, for any $c>0$, since $I_V(u_n)\rightarrow i_V(c)$, by \eqref{1.5} we have for $n$ large enough
 \begin{equation}\label{3.2}
\begin{array}{ll}
\ds i_V(c)+1+\frac{c^{\frac{2N-p(N-2)}{2}}}{2|Q_p|_2^{p-2}}\left(\int_{\R^N}|\nabla u_n|^2\right)^{\frac{N(p-2)}{4}}
&\geq\ds I_V(u_n)+\frac{c^{\frac{2N-p(N-2)}{2}}}{2|Q_p|_2^{p-2}}\left(\int_{\R^N}|\nabla u_n|^2\right)^{\frac{N(p-2)}{4}}\\[5mm]
&\geq\ds\frac{b}{4}\left(\int_{\R^N}|\nabla u_n|^2\right)^2+\frac12\int_{\R^N}V(x)u_n^2,
\end{array}
\end{equation}
which and $0<\frac{N(p-2)}{4}<2$ imply that $\{u_n\}$ is uniformly bounded in $\mathcal{H}$.

If $p=\frac{2N+8}{N}$, since $0<c\leq(\frac{b}{2}|Q_{\frac{2N+8}{N}}|_2^{\frac{8}{N}})^{\frac{N}{8-2N}},$ we deduce from \eqref{1.5} that $$\frac{N}{2N+8}\int_{\R^N}|u_{n}|^{\frac{2N+8}{N}}\leq \frac{b}{4}\left(\int_{\R^N}|\nabla u_{n}|^2\right)^2.$$
Then for $n$ large enough, we see that
$$i_V(c)+1\geq I_V(u_{n})\geq\frac{a}{2}\int_{\R^N}|\nabla u_{n}|^2+\frac12\int_{\R^N}V(x)u_n^2,$$
which implies that $\{u_n\}$ is uniformly bounded in $\mathcal{H}$.

Since $i_V(c)$ has a minimizer $u_c$, by the Lagrange multiplier theorem there exists $\rho_c\in\R$ such that $(u_c,\rho_c)\in S_c\times\R$ is a couple of solution to the equation \eqref{1.1}.

(2)~~Let $\varphi\in C_0^{\infty}(\R^N)$ be a radial cut-off function such that $0\leq\varphi\leq1$, $\varphi\equiv1$ on $B_1(0)$, $\varphi\equiv0$ on $\R^N\backslash B_2(0)$ and $|\nabla\varphi|\leq2$.

For each $c>(\frac{b}{2}|Q_{\frac{2N+8}{N}}|_2^{\frac{8}{N}})^{\frac{N}{8-2N}}$ if $p=\frac{2N+8}{N}$ or $c>0$ if $\frac{2N+8}{N}<p<2^*$, set
\begin{equation}\label{3.24}
u_{t}(x):=\frac{c A_{t}}{|Q_{p}|_2}\varphi(x)t^{\frac{N}{2}}Q_{p}(tx),~~~~\forall~t>0,
\end{equation}
where $A_t>0$ is chosen to satisfy that $u_{t}\in S_c$. In fact, By direct calculations and the Dominated Convergence Theorem  we have $$\frac{c^2}{A_{t}^2}=c^2+\frac{c^2}{|Q_{p}|_2^2}
\int_{\R^N}[\varphi^2(\frac{x}{t})-1]Q_{p}^2(x)\rightarrow c^2~~~\hbox{as}~t\rightarrow+\infty,$$
 i.e. $\lim\limits_{t\rightarrow+\infty}A_t=1.$ Since $V(x)\varphi^2(x)$ is bounded and has compact support, we see that $\lim\limits_{t\rightarrow+\infty}\int_{\R^N}V(x)|u_{t}|^2=V(0)c^2$. So by \eqref{2.3} and the exponential decay of $Q_{p}$ and $|\nabla Q_{p}|$,  we see that if $p=\frac{2N+8}{N},$ then
 $$
 \begin{array}{ll}
i_V(c)&\leq I_V(u_t)\\[5mm]
&=\ds\frac{ac^2A_t^2}{2}t^2-\Big{[}\Big{(}\frac{c}{(\frac{b}{2}|Q_{\frac{2N+8}{N}}|_2^{\frac{8}{N}})^{\frac{N}{8-2N}}}\Big{)}^{\frac{8-2N}{N}}-1+o_t(1)\Big{]}\frac{bc^4A_t^4}{4}t^4+\ds\frac{V(0)c^2}{2}+o_t(1)\\[5mm]
 &\rightarrow-\infty~~~~\hbox{as}~t\rightarrow+\infty
 \end{array}
 $$
 and if $\frac{2N+8}{N}<p<2^*$, then
 $$
 i_V(c)\leq I_V(u_t)=\ds\frac{a(cA_t)^2 t^2}{2}+\frac{b(cA_t)^4t^4}4-\frac{(cA_t)^p t^{\frac{N(p-2)}2}}{2|Q_p|_2^{p-2}}+\frac{V(0)c^2}{2}+o_t(1)\rightarrow-\infty~~\hbox{as}~t\rightarrow+\infty,
$$
since $\frac{N(p-2)}{2}>4$, where $\lim\limits_{t\rightarrow+\infty}o_t(1)=0$. Then $i_V(c)=-\infty$. So the theorem is proved.
\end{proof}

Next we give some preliminary lemmas to prove Theorem \ref{th1.3}.

\begin{lemma}\label{lem3.3}~~Let $2<p<\frac{2N+8}{N}$, $c>0$ and $V(x)$ satisfy $(V)$.

(1)~~The function $c\mapsto i_V(c)$ is continuous on $(0,+\infty).$

(2)~~$i_V(c)\rightarrow 0$ as $c\rightarrow 0^+$.

(3)~~If $p\neq 4$, then $\frac{i_V(c)}{i_0(c)}\rightarrow1$ as $c\rightarrow+\infty$.
\end{lemma}
\begin{proof}(1)~~For any $c>0$ and any sequence $\{c_n\}\subset(0,+\infty)$ satisfying that $c_n\rightarrow c$ as $n\rightarrow+\infty$, let $u_c\in S_c$ be a minimizer of $i_V(c)$, then $\frac{c_n}{c}u_c\in S_{c_n}$ and $i_V(c_n)\leq I_V(\frac{c_n}{c}u_c)=I_V(u_c)+o_n(1),$
where $o_n(1)\rightarrow0$ as $n\rightarrow+\infty$. So \begin{equation}\label{3.3}
\limsup\limits_{n\rightarrow+\infty}i_V(c_n)\leq i_V(c).
\end{equation}

On the other hand, let $\{u_{c_n}\}\subset S_{c_n}$ be a sequence of minimizers for $i_V(c_n)$. Since \eqref{3.3} shows that $\{i_V(c_n)\}$ is upper bounded, similarly to the proof of \eqref{3.2} we see that $\{u_{c_n}\}$ is uniformly bounded in $\mathcal{H}$. Then
$i_V(c)\leq \liminf\limits_{n\rightarrow+\infty}I_V(\frac{c}{c_n}u_{c_n})
=\liminf\limits_{n\rightarrow+\infty}i_V(c_n).$
So by \eqref{3.3} again we see that $\lim\limits_{n\rightarrow+\infty}i_V(c_n)=i_V(c).$

(2)~~Note that $i_V(c)\geq i_0(c)$ for any $c>0$ when $2<p<\frac{2N+8}{N}$. For $c=1$, suppose that $u_1$ is a minimizer of $i_V(1)$, then for any $c\rightarrow0^+$, $cu_1\in S_c$ and
\begin{equation}\label{3.1}
i_0(c)\leq i_V(c)\leq I_V(cu_1)\rightarrow0~~~~~\hbox{as}~~c\rightarrow 0^+.
\end{equation}
By Proposition \ref{lem1.1} we have to discuss the following two cases.

Case 1: $2<p<\frac{2N+4}{N}$.

By Proposition \ref{lem1.1}, if $2<p<\frac{2N+4}{N}$, then $i_0(c)<0$ and $i_0(c)$ has a minimizer $v_c\in \widetilde{S}_c$ for all $c>0$. Hence by \eqref{1.5} we see that $$\frac{a}2\int_{\R^N}|\nabla v_c|^2+\frac{b}{4}\left(\int_{\R^N}|\nabla v_c|^2\right)^2\leq\frac{c^{\frac{2N-p(N-2)}{2}}}{2|Q_p|_2^{p-2}}\left(\int_{\R^N}|\nabla v_c|^2\right)^{\frac{N(p-2)}{4}},$$
which and $0<\frac{N(p-2)}{4}<1$ imply that $\lim\limits_{c\rightarrow0^+}|\nabla v_c|_2=0.$ Then $\lim\limits_{c\rightarrow0^+}i_0(c)=0.$ So combining \eqref{3.1} we have $\lim\limits_{c\rightarrow0^+}i_V(c)=0.$

Case 2: $\frac{2N+4}{N}\leq p<\frac{2N+8}{N}$.

By Proposition \ref{lem1.1}, if $\frac{2N+4}{N}\leq p<\frac{2N+8}{N}$, then $i_0(c)=0$ for $c>0$ small. Then we conclude from \eqref{3.1} that $\lim\limits_{c\rightarrow0^+}i_V(c)=0.$

(3)~~For any $c>0$, we have $i_V(c)-i_0(c)\geq0$. Since $c\rightarrow+\infty$, we may assume that $c>c_*$, then $i_0(c)<0$ and $i_0(c)$ has a unique minimizer for all $2<p<\frac{2N+8}{N}$.

Let $\varphi\in C_0^{\infty}(\R^N)$ be a radial cut-off function as given in the proof of Theorem \ref{th1.1} (2). For any $x_0\in\R^N$ and $R>0$, Set
$$u_{R,c}(x):=A_{R,c}\varphi(\frac{x-x_0}{R})Q_c(x-x_0),$$
where $Q_c(x)=\frac{c}{|Q_p|_2}Q_p^{\frac{m_c}{c}}(x)$ is the unique minimizer of $i_0(c)$ given in Theorem \ref{th1.2} and $A_{R,c}>0$ is chosen to satisfy that $u_{R,c}\in S_c$. In fact, $0<A_{R,c}\leq 1$ and $\lim\limits_{R\rightarrow+\infty}A_{R,c}=1$. It is standard to show that $u_{R,c}\rightarrow Q_c$ in $H^1(\R^N)$ as $R\rightarrow+\infty$ for each $c>c_*$. Then by continuity, we see that for any $c>c_*$,
\begin{equation}\label{3.17}
I_0(u_{R,c})\rightarrow I_0(Q_c)=i_0(c)~~~\hbox{as} ~R\rightarrow+\infty.
\end{equation}
For any fixed $R>0$, since $p\neq4$, by Lemma \ref{lem4.5} (2) we see that either $\lim\limits_{c\rightarrow+\infty}\frac{m_c}{c}=+\infty$ or $\lim\limits_{c\rightarrow+\infty}\frac{m_c}{c}=0$.

If $\lim\limits_{c\rightarrow+\infty}\frac{m_c}{c}=+\infty$, then we have
$$
\begin{array}{ll}
0\leq\ds\int_{\R^N}V(x)u_{R,c}^2&=\ds
\frac{c^2A_{R,c}^2}{|Q_p|_2^2}\ds\int_{\R^N}V(\frac{c}{m_c}x+x_0)\varphi^2(\frac{c}{m_cR}x)Q_p^2(x)\\[5mm]
&\leq\ds
\frac{c^2}{|Q_p|_2^2}\ds\int_{\R^N}V(\frac{c}{m_c}x+x_0)\varphi^2(\frac{c}{m_cR}x)Q_p^2(x).
\end{array}
$$
We see that $0\leq\lim\limits_{c\rightarrow+\infty}\ds\frac{\int_{\R^N}V(x)u_{R,c}^2}{c^2}\leq V(x_0)$ holds for almost every $x_0\in\R^N$. Taking the infimum over $x_0$ yields \begin{equation}\label{3.18}
\ds\frac{\int_{\R^N}V(x)u_{R,c}^2}{c^2}\rightarrow0~~~~~\hbox{as}~c\rightarrow+\infty.
\end{equation}

If $\lim\limits_{c\rightarrow+\infty}\frac{m_c}{c}=0$, then
\begin{equation}\label{3.19}
\begin{array}{ll}
\ds\frac{\int_{\R^N}V(x)u_{R,c}^2}{c^2}&=\ds
\frac{A_{R,c}^2}{|Q_p|_2^2}(\frac{m_c}{c})^N\ds\int_{\R^N}V(x+x_0)\varphi^2(\frac{x}{R})Q_p^2(\frac{m_c}{c}x)\\[5mm]
&\leq \ds
\frac{1}{|Q_p|_2^2}(\frac{m_c}{c})^N\ds\int_{B_{2R}(0)}V(x+x_0)Q_p^2(\frac{m_c}{c}x)\\[5mm]
&\leq \ds
\frac{|V(x)|_{L^{\infty}(B_{2R}(x_0))}}{|Q_p|_2^2}(\frac{m_c}{c})^N\ds\int_{B_{2R}(0)}Q_p^2(\frac{m_c}{c}x)\\[5mm]
&\rightarrow0~~~~\hbox{as}~c\rightarrow+\infty,
\end{array}\end{equation}
where we have used the fact that $Q_p(x)$ is continuous.

Therefore for any $\varepsilon>0$, by \eqref{3.17}-\eqref{3.19} there exists $R_\varepsilon>0$ large enough such that
$$\begin{array}{ll}
\ds\frac{i_V(c)-i_0(c)}{c^2}&\leq \ds \frac{I_V(u_{R_\varepsilon,c})}{c^2}-\frac{i_0(c)}{c^2}\\[5mm]
&\ds=\frac{I_0(u_{R_\varepsilon,c})}{c^2}-\frac{i_0(c)}{c^2}+\ds\frac{\int_{\R^N}V(x)u_{R_\varepsilon,c}^2}{2c^2}\leq \varepsilon
\end{array}$$
for sufficiently large $c>0.$ So $\lim\limits_{c\rightarrow+\infty}(\frac{i_V(c)}{c^2}-\frac{i_0(c)}{c^2})=0$, i.e. $$\lim\limits_{c\rightarrow+\infty}\left(\frac{i_V(c)}{i_0(c)}-1\right)\frac{i_0(c)}{c^2}=0.$$ We conclude from Lemma \ref{lem4.5} (1) that $\frac{i_V(c)}{i_0(c)}\rightarrow1$ as $c\rightarrow+\infty$.
\end{proof}

\begin{corollary}\label{cor3.4}~~Let $2<p<\frac{2N+8}{N}$ and $p\neq4$.

(1)~~
$\frac{i_V(c)}{c^2}\rightarrow-\infty$ as $c\rightarrow+\infty$.

~~~~~~~$i_V(c)\rightarrow-\infty$ as $c\rightarrow+\infty$.

(2)~~Suppose that $u_c$ is a minimizer of $i_V(c)$ obtained in Theorem \ref{th1.1}. Then $\int_{\R^N}|u_c|^p\rightarrow+\infty$ and $\int_{\R^N}|\nabla u_c|^2\rightarrow+\infty$ as $c\rightarrow+\infty$.
\end{corollary}
\begin{proof}
(1) follows directly from Lemma \ref{lem3.3} and Lemmas \ref{lem4.5}, \ref{lem4.4}.

(2)~~Since $I_V(u_c)=i_V(c)$, by the definition of $I_V(u_c)$ we have $\frac{1}{p}\int_{\R^N}|u_c|^{p}+i_V(c)>0$, which and (1) implies that $\int_{\R^N}|u_c|^p\rightarrow+\infty$ as $c\rightarrow+\infty$. Moreover, we have $\frac{\frac{1}{p}\int_{\R^N}|u_c|^{p}+i_V(c)}{c^2}>0$, then $\frac{\frac{1}{p}\int_{\R^N}|u_c|^{p}}{c^2}\rightarrow+\infty$. By the Gagliardo-Nirenberg inequality \eqref{1.5}, we see that $\frac{\frac{1}{p}\int_{\R^N}|u_c|^{p}}{c^2}\leq \frac{p}{2|Q_p|_2^{p-2}}\left(\frac{|\nabla u_c|_2}{c}\right)^{\frac{(N-2)(p-2)}{2}}|\nabla u_c|_2^{p-2},$
from which we conclude that
$\int_{\R^N}|\nabla u_c|^2\rightarrow+\infty$ as $c\rightarrow+\infty$.
\end{proof}

In order to prove Theorem \ref{th1.3}, we need the following two crucial lemmas.

\begin{lemma}\label{lem3.5}~~Let $p\in(2,\frac{2N+8}{N})\backslash\{4\}$, $c>0$ and $V(x)$ satisfy $(V)$. Suppose that $u_c\in S_c$ be a minimizer of $i_V(c)$ obtained in Theorem \ref{th1.1}, then
$$\frac{\int_{\R^N}|\nabla u_c|^2}{m_c^2}\rightarrow \frac{a}{b},$$
$$\frac{\int_{\R^N}|u_c|^{p}}{am_c^2+bm_c^4}\rightarrow\frac{2p}{N(p-2)} $$
and $$\frac{\int_{\R^N}V(x)u_c^2}{am_c^2+bm_c^4}\rightarrow0$$
as $c\rightarrow+\infty,$ where $m_c$ is given in Theorem \ref{th1.2}.
\end{lemma}
\begin{proof}~~As $c\rightarrow+\infty$, for $c>0$ sufficiently large, by the Gagliardo-Nirenberg inequality \eqref{1.5} and \eqref{4.6} we have
\begin{equation}\label{3.4}
\begin{array}{ll}
\ds 0<\frac{\frac{1}{p}\int_{\R^N}|u_c|^{p}}{(a+bm_c^2)\int_{\R^N}|\nabla u_c|^2}
&\leq \ds\frac{\frac{c^{p-\frac{N(p-2)}{2}}}{2|Q_p|_2^{p-2}}\left(\int_{\R^N}|\nabla u_c|^2\right)^{\frac{N(p-2)-4}{4}}}{a+bm_c^2}\\[5mm]
&\leq \ds\frac{\frac{c^{p-\frac{N(p-2)}{2}}}{2|Q_p|_2^{p-2}}\left(\int_{\R^N}|\nabla u_c|^2\right)^{\frac{N(p-2)-4}{4}}}{\frac{N(p-2)}{4}\frac{c^{p-\frac{N(p-2)}{2}}}{|Q_p|_2^{p-2}}m_c^{\frac{N(p-2)-4}{2}}}\\[5mm]
&\leq \ds\frac{2}{N(p-2)}\left(\frac{\int_{\R^N}|\nabla u_c|^2}{m_c^2}\right)^{\frac{N(p-2)-4}{4}}.
\end{array}
\end{equation}
Then we see that
\begin{equation}\label{3.5}
\begin{array}{ll}
\ds\frac{I_V(u_c)}{(a+bm_c^2)\int_{\R^N}|\nabla u_c|^2}
&=\ds\frac a{2(a+bm_c^2)}+\frac{b\int_{\R^N}|\nabla u_c|^2}{4(a+bm_c^2)}+\frac{\frac12\int_{\R^N}V(x)u_c^2}{(a+bm_c^2)\int_{\R^N}|\nabla u_c|^2}\\[5mm]
&~~~~~~~~~~~~~~~~~~~~~~~~~~~~~\ds-\frac{\frac{1}{p}\int_{\R^N}|u_c|^{p}}{(a+bm_c^2)\int_{\R^N}|\nabla u_c|^2}\\[5mm]
&\geq\ds\frac{b\int_{\R^N}|\nabla u_c|^2}{4(a+bm_c^2)}-\frac{2}{N(p-2)}\left(\frac{\int_{\R^N}|\nabla u_c|^2}{m_c^2}\right)^{\frac{N(p-2)-4}{4}}\\[5mm]
&\geq\ds-\frac{2}{N(p-2)}\left(\frac{\int_{\R^N}|\nabla u_c|^2}{m_c^2}\right)^{\frac{N(p-2)-4}{4}}+\frac{\int_{\R^N}|\nabla u_c|^2}{m_c^2}\frac{bm_c^2}{4(a+bm_c^2)}.
\end{array}
\end{equation}
Moreover, by $I_V(u_c)=i_V(c)$ and Corollary \ref{cor4.3} (1) we have
\begin{equation}\label{3.20}
\begin{array}{ll}
\ds\frac{I_V(u_c)}{(a+bm_c^2)\int_{\R^N}|\nabla u_c|^2}
&=\ds\frac{i_V(c)}{i_0(c)}\frac{m_c^2}{\int_{\R^N}|\nabla u_c|^2}\frac{aD_1-bD_2m_c^2}{a+bm_c^2}.
\end{array}
\end{equation}

Up to a subsequence, we assume that $\frac{\int_{\R^N}|\nabla u_c|^2}{m_c^2}\rightarrow+\infty$ as $c\rightarrow+\infty,$ then by Lemmas \ref{lem3.3} and \ref{lem4.4}, \eqref{3.5} and \eqref{3.20} we see that
\begin{equation}\label{3.21}\begin{array}{ll}
0&=\lim\limits_{c\rightarrow+\infty}\ds\frac{I_V(u_c)}{(a+bm_c^2)\int_{\R^N}|\nabla u_c|^2}\\[5mm]
&\geq\ds\lim\limits_{c\rightarrow+\infty}\left\{\left[-\frac{2}{N(p-2)}\left(\frac{m_c^2}{\int_{\R^N}|\nabla u_c|^2}\right)^{\frac{2N+8-Np}{4}}+\frac{bm_c^2}{4(a+bm_c^2)}\right]\frac{\int_{\R^N}|\nabla u_c|^2}{m_c^2}\right\}\\[5mm]
&\rightarrow+\infty~~~~\hbox{as}~c\rightarrow+\infty
\end{array}
\end{equation}
since $2<p<\frac{2N+8}{N}$, which is a contradiction. Therefore, $\{\frac{\int_{\R^N}|\nabla u_c|^2}{m_c^2}\}$ is uniformly bounded. Set $$A:=\lim\limits_{c\rightarrow+\infty}\frac{\int_{\R^N}|\nabla u_c|^2}{m_c^2},$$
 then $A\in[0,+\infty)$. If $A=0$, then similarly to \eqref{3.21}, by $D_2>0$ we have
$$\begin{array}{ll}
0&=\lim\limits_{c\rightarrow+\infty}\ds\frac{bm_c^2}{4(a+bm_c^2)}\frac{\int_{\R^N}|\nabla u_c|^2}{m_c^2}\\[5mm]
&\leq \lim\limits_{c\rightarrow+\infty}\ds\frac{I_V(u_c)}{(a+bm_c^2)\int_{\R^N}|\nabla u_c|^2}+\frac{2}{N(p-2)}\left(\frac{\int_{\R^N}|\nabla u_c|^2}{m_c^2}\right)^{\frac{N(p-2)-4}{4}}\\[5mm]
&\leq\ds\lim\limits_{c\rightarrow+\infty}\left\{\left[\ds\frac{i_V(c)}{i_0(c)}\frac{-bD_2m_c^2+aD_1}{a+bm_c^2}+\frac{2}{N(p-2)}\left(\frac{\int_{\R^N}|\nabla u_c|^2}{m_c^2}\right)^{\frac{N(p-2)}{4}}\right]\frac{m_c^2}{\int_{\R^N}|\nabla u_c|^2}\right\}\\[5mm]
&\rightarrow-\infty~~~~\hbox{as}~c\rightarrow+\infty,
\end{array} $$ which is also a contradiction. So $0<A<+\infty$. Moreover, set
$$B:=\lim\limits_{c\rightarrow+\infty}\ds\frac{\frac{1}{p}\int_{\R^N}|u_c|^{p}}{(a+bm_c^2)\int_{\R^N}|\nabla u_c|^2}$$
and
$$D:=\lim\limits_{c\rightarrow+\infty}\frac{\frac12\int_{\R^N}V(x)u_c^2}{(a+bm_c^2)\int_{\R^N}|\nabla u_c|^2},$$
then it follows from \eqref{3.4}-\eqref{3.20} that $B\in(0,\frac{2}{N(p-2)}A^{\frac{N(p-2)-4}{4}}]$ and $D\in[0,+\infty)$.
We conclude from \eqref{3.5}\eqref{3.20} again that $A$ satisfies the following two conditions:
\begin{equation}\label{3.6}
\frac14A^2+(D-B)A+D_2=0
\end{equation}
and
\begin{equation}\label{3.22}
 \frac14A^2-\frac{2}{N(p-2)}A^{\frac{N(p-2)}{4}}+D_2\leq0.
\end{equation}

 Consider the following $C^1$-function $f:[0,+\infty)\rightarrow\R$ defined by $$f(t)=\frac14t^2-\frac{2}{N(p-2)}t^{\frac{N(p-2)}{4}}+D_2.$$
 Since $2<p<\frac{2N+8}{N}$ and $f'(t)=\frac12(t-t^{\frac{N(p-2)-4}{4}})$, $f(t)$ has a unique critical point $t=1$ which corresponds to its minimum, i.e. $f(t)> f(1)=0$ for all $t\neq1$. So combining \eqref{3.22} we see that $f(A)=0,$ which implies that $A=1$. Then by \eqref{3.6} and the definition of $D_2$ we have $$\frac{2}{N(p-2)}\leq D+\frac{2}{N(p-2)}=B\leq\frac{2}{N(p-2)},$$ then $D=0$ and hence $B=\frac{2}{N(p-2)}$. Therefore the lemma is proved.
\end{proof}

\begin{lemma}\label{lem3.6}~~ For $2<p<\frac{2N+8}{N}$ and $p\neq4$, let $u_c\in S_c$ be a minimizer of $i_V(c)$. Then $\frac{I_0(u_c)}{i_0(c)}\rightarrow1$
as $c\rightarrow+\infty$.
\end{lemma}
\begin{proof}~~
As $c\rightarrow+\infty$, we may assume that $c>c_*$. Then by Corollary \ref{cor4.3} (1), Lemmas \ref{lem3.5} and \ref{lem4.4}, we see that
$$\begin{array}{ll}
\ds\frac{I_0(u_c)}{i_0(c)}&=\ds\frac{\frac{1}{p}\int_{\R^N}|u_c|^{p}}{bD_2m_c^4 -aD_1m_c^2}-\frac{a\frac{\int_{\R^N}|\nabla u_c|^2}{m_c^2}}{2(bD_2m_c^2 -aD_1)}-\frac{b\left(\frac{\int_{\R^N}|\nabla u_c|^2}{m_c^2}\right)^2}{4(bD_2 -aD_1m_c^{-2})}\\[5mm]
&\rightarrow \ds\left(\frac{2}{N(p-2)}-\frac{1}{4}\right)\frac{1}{D_2}=1
\end{array}$$
as $c\rightarrow+\infty$.
\end{proof}

\noindent $\textbf{Proof of Theorem \ref{th1.3}}$\,\,\

\begin{proof}~~For any sequence $\{c_n\}\subset(0,+\infty)$ with $c_n\rightarrow+\infty$ as $n\rightarrow+\infty$, by Theorem \ref{th1.1}, there exists $(u_{c_n},\rho_{c_n})\in S_{c_n}\times\R$ a couple of solution to the equation \eqref{1.1} with $I_V(u_{c_n})=i_V(c_n)$. Then $u_{c_n}$ is a minimizer of $i_V(c_n).$

Without loss of generality, we may assume that $c_n>c_*$ for all $n$. By Lemma \ref{lem3.5} we see that
\begin{equation}\label{3.7}
\frac{\int_{\R^N}|\nabla u_{c_n}|^2}{m_{c_n}^2}\rightarrow 1~~~~~~~\hbox{and}~~~~~~~\frac{\int_{\R^N}|u_{c_n}|^{p}}{am_{c_n}^2+bm_{c_n}^4}\rightarrow\frac{2p}{N(p-2)} \end{equation}
 as $n\rightarrow+\infty.$

 Since $\rho_{c_n}{c_n^2}=\langle I'_V(u_{c_n}),u_{c_n}\rangle,$ by Corollary \ref{cor4.3} (2), \eqref{4.6} and Lemma \ref{lem3.5} we have
$$\begin{array}{ll}
\ds\frac{\rho_{c_n}}{\mu_{c_n}}
&=\ds-\frac{\rho_{c_n}c_n^2}{\frac{2N-p(N-2)}{4}\frac{{c_n}^{\frac{2N-p(N-2)}{2}}}{|Q_p|_2^{p-2}}m_{c_n}^{\frac{N(p-2)}{2}}}\\[5mm]
&=\ds-\frac{N(p-2)}{2N-p(N-2)}\frac{\langle I'_V(u_{c_n}),u_{c_n}\rangle}{am_{c_n}^2+bm_{c_n}^4}\\[5mm]
&=\ds\frac{N(p-2)}{2N-p(N-2)}\frac{\int_{\R^N}|u_{c_n}|^{p}-a\int_{\R^N}|\nabla u_{c_n}|^2-b(\int_{\R^N}|\nabla u_{c_n}|^2)^2-\int_{\R^N}V(x)u_{c_n}^2}{am_{c_n}^2+bm_{c_n}^4}\\[5mm]
&\rightarrow\ds\frac{N(p-2)}{2N-p(N-2)}\left(\frac{2p}{N(p-2)}-1\right)=1
\end{array}
$$
as $c\rightarrow+\infty$.

For any $c\in T$, set
$$v_{n}(x):=\frac{c}{c_n}u_{c_n}^{\frac{c_n m_c}{m_{c_n} c}}(x),$$
where $m_{c}=(\frac{\sqrt{a^2D_1^2-4bD_2i_0(c)}+aD_1}{2bD_2})^{\frac12}$ is given in Theorem \ref{th1.2}. Then $v_n\in \widetilde{S}_{c}$ and by \eqref{3.7} we see that
\begin{equation}\label{3.8}
\int_{\R^N}|\nabla v_n|^2=m_c^2\frac{\int_{\R^N}|\nabla u_{c_n}|^2}{m_{c_n}^2}\rightarrow m_c^2
\end{equation}
and by \eqref{4.6}\eqref{3.7} we have
\begin{equation}\label{3.9}
\begin{array}{ll}
\ds\int_{\R^N}|v_n|^{p}
&=\ds c^{p-\frac{N(p-2)}{2}}m_c^{\frac{N(p-2)}2}\frac{\int_{\R^N}|u_{c_n}|^{p}}{c_n^{p-\frac{N(p-2)}{2}}m_{c_n}^{\frac{N(p-2)}{2}}}\\[5mm]
&=\ds\frac{N(p-2)|Q_p|_2^{2-p}c^{p-\frac{N(p-2)}{2}}m_c^{\frac{N(p-2)}2}}{4}\frac{\int_{\R^N}|u_{c_n}|^{p} }{(a+bm_{c_n}^2)m_{c_n}^2}\\[5mm]
&\rightarrow \ds\frac{pc^{p-\frac{N(p-2)}{2}}m_c^{\frac{N(p-2)}2}}{2|Q_p|_2^{p-2}}
=\frac{2p(am_c^2+bm_c^4)}{N(p-2)}.
\end{array}
\end{equation}
as $n\rightarrow+\infty$. Hence
$$I_0(v_n)\rightarrow \frac{a}{2}m_c^2+\frac{b}{4}m_c^4-\frac{2(am_c^2+bm_c^4)}{N(p-2)}=aD_1m_c^2-bD_2m_c^4=i_0(c),$$ i.e. $\{v_n\}$ is a bounded minimizing sequence for $i_0(c)$. So by Proposition \ref{pro4.6}, up to a subsequence and up to translations, there exists $\{y_n\}\subset\R^N$ such that $v_n(\cdot+y_n)\rightarrow \frac{c}{|Q_p|_2}Q_p^{\frac{m_c}{c}}$ in $H^1(\R^N)$ as $n\rightarrow+\infty$. So by the embedding theorem, we see that
 $$\frac{|Q_{p}|_2}{c_n}\left(\frac{c_n}{m_{c_n}}\right)^{\frac{N}{2}}\
 u_{c_n}\left(\frac{c_n}{m_{c_n}}(x+z_n)\right)\longrightarrow Q_p(x)$$
 in $L^q(\R^N)$ for all $2\leq q<2^*$, where $z_n=\frac{m_c}{c}y_n.$

\end{proof}

\textbf{Acknowledgments:} After our paper is submitted, Zeng X. Y. and Zhang Y. M. publish a paper \cite{zz} which also studies the uniqueness of mass minimizers for \eqref{1.8} by an alternative method, however, our uniqueness theorem (Theorem \ref{th1.2}) is more delicate.



 \end{document}